\documentclass[a4paper]{amsart}
\usepackage[margin=3.5cm]{geometry}
\usepackage[T1]{fontenc}
\usepackage[utf8]{inputenc}
\usepackage{amsmath,amsfonts,amssymb,amsthm}
\usepackage{mathtools,bbold}
\usepackage{tikz-cd}
\usepackage{comment}
\usepackage{enumitem}
\usepackage[bookmarksdepth=2]{hyperref}
\usepackage{leftidx}
\usepackage{ifthen}

\theoremstyle{plain}
\newtheorem{Theorem}{Theorem}[section]

\newtheorem{Proposition}[Theorem]{Proposition}
\newtheorem{Corollary}[Theorem]{Corollary}
\newtheorem{Lemma}[Theorem]{Lemma}

\theoremstyle{definition}

\theoremstyle{remark}

\newtheorem{Example}[Theorem]{Example}

\numberwithin{equation}{section}

\newcommand{\stackcite}[1]{\cite[\href{http://stacks.math.columbia.edu/tag/#1}{Tag~#1}]{stacks-project}}

\mathchardef\mhyphen="2D    

\newcommand\Hom{\operatorname{Hom}}

\newcommand\id{\operatorname{id}}
\DeclareMathOperator\colim{colim}
  \DeclareMathOperator{\modd}{Mod}                    
  
  \DeclareMathOperator{\hproj}{P}

  \DeclareMathOperator{\acyc}{Ac}

\newcommand\numGgp[1]{K_0^{\mathrm{num}}\ifthenelse{\equal{#1}{}}{}{(#1)}}      
\newcommand\rightmod[1]{{\modd\mkern-2mu\mhyphen #1}}   
\newcommand\leftmod[1]{{#1\mhyphen\mkern-2.5mu\modd}}   
\newcommand\bimod[2]{{#1\mhyphen\mkern-2.5mu\modd\mkern-2mu\mhyphen #2}}   

\DeclareMathOperator{\cok}{Coker}
\DeclareMathOperator{\homm}{Hom}
\DeclareMathOperator{\opp}{{opp}}
\DeclareMathOperator{\action}{act}

\author{Ádám Gyenge}
\address{Budapest University of Technology and Economics, Department of Algebra and Geometry, M\H{u}egyetem rakpart 3-9., 1111, Budapest, Hungary}
\email{Gyenge.Adam@ttk.bme.hu}

\title{On a sequence of Grothendieck groups}

\subjclass[2020]{Primary 18G80; Secondary 18F25}
\keywords{numerical K-group, recollement, Drinfeld quotient, Calabi-Yau triple}

\begin{document}

\begin{abstract}
We show that a well-known exact sequence in K-theory for quotients of triangulated categories descends to numerical K-groups provided that the category, the quotient and the category we take the quotient with has a numerical K-group, and if either the quotient functor preserves compactness or the K-group of the quotient is torsion-free.
\end{abstract}

\maketitle

\section{Introduction}
\label{sec:intro}

Let ${\mathcal T}$ be a triangulated category, $\mathcal{S}$ a subcategory and let ${\mathcal T}/\mathcal{S}$ be the Verdier quotient. Then there is an exact sequence of ordinary Grothendieck groups \cite[Proposition~VIII.3.1.]{grothendieck1977cohomologie}:
\begin{equation}  
\label{eq:kgrpseq}
K_0(\mathcal{S}) \xrightarrow{i^\ast} K_0({\mathcal T}) \xrightarrow{q^\ast} K_0({\mathcal T}/\mathcal{S})  \to 0
\end{equation}
where $i^{\ast}$ is induced by the embedding functor $i: \mathcal{S} \subset {\mathcal T}$ and $q^{\ast}$ is induced by the quotient functor $q: {\mathcal T} \to {\mathcal T}/\mathcal{S}$.

The numerical Grothendieck group is the quotient of the usual Grothendieck group by the kernel (or radical) of the Euler form. When exists, it often gives a more tractable invariant than the classical K-group. For example, when ${\mathcal T}$, or more precisely its dg enhancement, is smooth and proper, then the numerical Grothendieck group is a finitely generated free abelian group. 

In this paper we show that the sequence \eqref{eq:kgrpseq} descends to numerical Grothendieck groups under quite general circumstances. 

\begin{Theorem}[{Theorem~\ref{lem:numgrpseq}}]
\label{thm:main}
Let ${\mathcal T}$ be a triangulated category and $\mathcal{S}$ a strictly full triangulated subcategory of ${\mathcal T}$. Assume that the numerical Grothendieck groups of ${\mathcal T}, \mathcal{S}$ and ${\mathcal T}/\mathcal{S}$ all exist.
Suppose moreover that $q$ or, equivalently, $i$ preserves compactness. 
Then there is an exact sequence
\[\numGgp{ \mathcal{S}} \to \numGgp{ {\mathcal T}} \to \numGgp{{\mathcal T}/\mathcal{S}} \to 0\]
of numerical Grothendieck groups.
\end{Theorem}

In the main body of the paper, for simplicity, we work in the context of derived categories of DG categories. However, in the proof of Theorem~\ref{thm:main} we do not actually refer to the DG enhancements, so the claims hold in the above generality provided the necessary functors exist. 

We discuss first a recollement arising from quotients. For derived categories of (quasi-compact and quasi-separated) schemes this recollement has appeared in \cite{jorgensen2009new}. For DG categories the statements are also known to experts as they follow implicitly from \cite[Proposition~4.6~(ii)]{drinfeld2004dg}. In the background, recollements are closely related to semiorthogonal decompositions, and DG quotients induce such a decomposition under quite general circumstances. Our contribution is the extension of the methods of \cite{jorgensen2009new} to the DG setting, hence giving an alternative proof of \cite[Proposition~4.6~(ii)]{drinfeld2004dg}, a detailed description of the relevant functors as well as an application of the results on numerical K-groups.

After preliminary definitions and a review of the basic structures we show:
\begin{itemize}
	\item the existence of a recollement on the triple of derived categories $(D({\mathcal V}/{\mathcal I}), D({\mathcal V}),D({\mathcal I}))$ for any strictly full DG subcategory ${\mathcal I}$ of a DG category ${\mathcal V}$ with Drinfeld quotient ${\mathcal V}/{\mathcal I}$; 
	\item the existence of a ``half recollement''  on the triple of the compact subcategories $(D_c({\mathcal V}/{\mathcal I}), D_c({\mathcal V}),D_c({\mathcal I}))$ whenever the inclusion or the quotient functor preserves compactness. 
\end{itemize}
This setting gives one of the examples where the conditions of Theorem~\ref{thm:main} hold (see Example~\ref{ex:smoothprop}).

Another situation is the following. Iyama-Yang introduced in \cite{iyama2020quotients} the notion of a \emph{relative Serre quadruple}. Such a quadruple $({\mathcal T},\mathcal{S},S,\mathcal{M})$ consists of a $k$-linear Hom-finite Krull–Schmidt triangulated category ${\mathcal T}$, a thick subcategory $\mathcal{S}$, a triangulated equivalence $S: \mathcal{S} \to \mathcal{S}$ such that for all $A \in \mathcal{Y}$, $B \in {\mathcal T}$
there is a bifunctorial isomorphism
\[ \Hom(A,B)^\ast \cong \Hom(B,SA), \]
and a silting subcategory $\mathcal{M}$ of ${\mathcal T}$ which gives rise to a t-structure on ${\mathcal T}$ via
\[ ({\mathcal T}^{\leq 0}, {\mathcal T}^{\geq 0}) \coloneqq ((\Sigma^{>0} \mathcal{M})^{\perp},(\Sigma^{<0} \mathcal{M})^{\perp})\]
with ${\mathcal T}^{\geq 0} \subset \mathcal{S}$.
DG algebras giving rise to such quadruples were investigated in \cite{guo2011cluster}. It was shown in \cite[Theorem~2.3.3]{dannetun2024quotients} that if $S$ extends to an autoequivalence of ${\mathcal T}$, then the quotient ${\mathcal T}/\mathcal{S}$ also has a Serre functor. In the special case when $S=[n]$ this condition always holds (see \cite[Theorem~4.10]{iyama2018silting}). Then
, by suppressing the third component in the tuple, $({\mathcal T},\mathcal{S},\mathcal{M})$ is called a \emph{Calabi-Yau triple}, and the quotient ${\mathcal T}/\mathcal{S}$ is called a \emph{cluster category}.
When the extended autoequivalence $S: {\mathcal T} \to {\mathcal T}$ is a Serre functor, e.g. if ${\mathcal T}$ itself is a Calabi-Yau category, the conditions of Theorem~\ref{thm:main} hold.

\subsection*{Acknowledgements}  
The author is thankful to Clemens Koppensteiner and Timothy Logvinenko for helpful discussions. This project has received funding from the European Union's Horizon 2020 research and innovation programme under the Marie Sk\l odowska-Curie grant agreement No 891437. The project was also supported by the János Bolyai Research Scholarship of the Hungarian Academy of Sciences and by the National Research, Development and Innovation Fund of Hungary, within the Program of Excellence TKP2021-NVA-02 at the Budapest University of Technology and Economics.

\section{Preliminaries}

We start with a review of the main notions we use.

\subsection{DG-categories and DG-modules}
\label{section-dg-categories-and-dg-modules}

Let ${\mathbb{k}}$ be a field.
A \emph{DG-category} 
${\mathcal{A}}$ is a category enriched over the monoidal category ${\rightmod{{\mathbb{k}}}}$ of
complexes of ${\mathbb{k}}$-modules, that is, a category where the morphism
spaces are objects of ${\rightmod{{\mathbb{k}}}}$ and the compositions are morphisms of dg
${\mathbb{k}}$-modules.

All functors from now on are assumed to be DG. 
A (right) module $E$ over ${\mathcal{A}}$ is a functor
$E \colon {{\mathcal{A}}^{\opp}} \rightarrow {\rightmod{{\mathbb{k}}}}$. For any $a \in {\mathcal{A}}$ we write $E_a$
for the complex $E(a) \in {\rightmod{{\mathbb{k}}}}$, the \emph{fiber of $E$ over $a$}. 
We write ${\rightmod {\mathcal{A}}}$ for the DG-category of (right) ${\mathcal{A}}$-modules. 
Similarly, a left ${\mathcal{A}}$-module $F$ is a functor $F\colon {\mathcal{A}} \rightarrow
{\rightmod{{\mathbb{k}}}}$. We write $\leftidx{_a}F$ for the fiber $F(a) \in {\rightmod{{\mathbb{k}}}}$
of $F$ over $a \in A$ and ${\leftmod {\mathcal{A}}}$ for the DG-category of left 
${\mathcal{A}}$-modules.  

Given another DG-category ${\mathcal{B}}$, an
${{\mathcal{A}}\mhyphen {\mathcal{B}}}$-bimodule $M$ is an ${{\mathcal{A}}^{\opp}} \otimes_{\mathbb{k}} {\mathcal{B}}$-module. 
For any $a \in {\mathcal{A}}$ and $b \in {\mathcal{B}}$ we write $\leftidx{_{a}}{M}{} \in {\rightmod {\mathcal{B}}}$
for the fiber $M(a,-)$ of $M$ over $a$, 
$M_b \in {\leftmod {\mathcal{A}}}$ for the fiber $M(-,b)$ of $M$ over $b$, 
and $\leftidx{_{a}}{M}{_{b}} \in {\rightmod{{\mathbb{k}}}}$ for the fiber of $M$ over $(a,b)$. 
We write ${\bimod{\mathcal{A}}{\mathcal{B}}}$ for the DG-category of ${\mathcal{A}}$-${\mathcal{B}}$-bimodules. 
The categories ${\rightmod {\mathcal{A}}}$ and ${\leftmod {\mathcal{A}}}$ of right and left ${\mathcal{A}}$-modules
can therefore be considered as the categories of ${\mathbb{k}}$-${\mathcal{A}}$- and 
${\mathcal{A}}$-${\mathbb{k}}$-bimodules. For any DG-category ${\mathcal{A}}$, we write ${\mathcal{A}}$ for 
the diagonal ${{\mathcal{A}}\mhyphen {\mathcal{A}}}$-bimodule defined by 
$\leftidx{_a}{\mathcal{A}}_b = \homm_{\mathcal{A}}(b,a)$ for all $a,b \in {\mathcal{A}}$ with 
morphisms of ${\mathcal{A}}$ and ${\mathcal{B}}$ acting by pre- and post-composition 
respectively: 
\begin{equation*}
{\mathcal{A}}(\alpha \otimes \beta) = 
(-1)^{\deg(\beta) \deg(-)} \alpha \circ (-) \circ \beta,
\quad \quad 
\forall\; \alpha \in \homm_{\mathcal{A}}(a,a'), \beta \in \homm_{\mathcal{A}}(b',b).
\end{equation*}

DG-bimodules over DG-categories admit 
a closed symmetric monoidal structure. That is, they admit 
a tensor product and an internal $\Hom$. Given three DG-categories 
${\mathcal{A}}$, ${\mathcal{B}}$ and ${\mathcal{C}}$, there exist functors 
\begin{equation*}
(-)\otimes_{\mathcal{B}}(-) \colon {\bimod{\mathcal{A}}{\mathcal{B}}} \otimes {\bimod{\mathcal{B}}{\mathcal{C}}} \rightarrow {\bimod{\mathcal{A}}{\mathcal{C}}},
\end{equation*}
\begin{equation*}
\homm_{\mathcal{B}}(-,-) \colon {\bimod{\mathcal{A}}{\mathcal{B}}} \otimes {\bimod{\mathcal{C}}{\mathcal{B}}} \rightarrow {\bimod{\mathcal{C}}{\mathcal{A}}},
\end{equation*}
\begin{equation*}
\homm_{\mathcal{B}}(-,-) \colon {\bimod{\mathcal{B}}{\mathcal{A}}} \otimes {\bimod{\mathcal{B}}{\mathcal{C}}} \rightarrow {\bimod{\mathcal{A}}{\mathcal{C}}},
\end{equation*}
where  
\[ M \otimes_{\mathcal{B}} N = 
\cok (M \otimes_{\mathbb{k}} {\mathcal{B}} \otimes_{\mathbb{k}} N \xrightarrow{\action \otimes \id -
\id \otimes \action} M \otimes_{\mathbb{k}} N), \]
while $\leftidx{_c}{\left(\homm_{\mathcal{B}}(M,N)\right)}_a = \homm_{\mathcal{B}}(\leftidx{_{a}}{M}{},\leftidx{_{c}}{N}{})$ 
for $M,N$ with right ${\mathcal{B}}$-action, and 
$\leftidx{_a}{\left(\homm_{\mathcal{B}}(M,N)\right)}_c = \homm_{\mathcal{B}}(M_c,N_a)$
for $M,N$ with left ${\mathcal{B}}$-action, 
cf.~\cite[Section~2.1.5]{AnnoLogvinenko-SphericalDGFunctors} for details. 

\subsection{The derived category of a DG-category}

Let ${\mathcal{A}}$ be a DG-category. Its \emph{homotopy category} $H^0({\mathcal{A}})$ is 
defined as the ${\mathbb{k}}$-linear category whose objects are the same as those of ${\mathcal{A}}$ and whose $\homm$ spaces are $H^0(-)$ 
of the $\homm$ complexes of ${\mathcal{A}}$. 

The category $H^0({\rightmod {\mathcal{A}}})$ has a natural structure of a triangulated
category defined levelwise in ${\rightmod{{\mathbb{k}}}}$. That is, the homotopy category 
$H^0({\rightmod{{\mathbb{k}}}})$ of complexes of ${\mathbb{k}}$-modules has a well-known
triangulated structure, and we define one on $H^0({\rightmod {\mathcal{A}}})$
by using that structure in the fibers over each $a \in {\mathcal{A}}$. 
A DG-category ${\mathcal{A}}$ is \emph{pretriangulated} if $H^0({\mathcal{A}})$ is a 
triangulated subcategory of $H^0({\rightmod {\mathcal{A}}})$ under the Yoneda embedding.

Given a DG-category ${\mathcal{A}}$ with a full subcategory ${\mathcal{C}} \subset {\mathcal{A}}$
there exists the \emph{Drinfeld quotient} ${\mathcal{A}} / {\mathcal{C}}$ (see \cite{drinfeld2004dg}). It is constructed by formally adding 
for every $c \in {\mathcal{C}}$ a degree $-1$ contracting homotopy $\xi_c$ with 
$d\xi_c = \id_c$ . When ${\mathcal{A}}$ and ${\mathcal{C}}$ are pretriangulated, one recovers
the Verdier quotient as $H^0({\mathcal{A}}/{\mathcal{C}}) \simeq H^0({\mathcal{A}})/H^0({\mathcal{C}})$. 

An ${\mathcal{A}}$-module $E$ is \emph{acyclic} if $E_a$ is an acyclic complex
for all $a \in {\mathcal{A}}$. We denote by $\acyc {\mathcal{A}}$ the full subcategory of 
${\rightmod {\mathcal{A}}}$ consisting of acyclic modules. A morphism of ${\mathcal{A}}$-modules is
a \emph{quasi-isomorphism} if it is a quasi-isomorphism levelwise
in ${\rightmod{{\mathbb{k}}}}$ for every $a \in {\mathcal{A}}$. The derived category $D({\mathcal{A}})$ is the 
localisation of $H^0({\rightmod {\mathcal{A}}})$ by quasi-isomorphisms, constructed as
the Verdier quotient $H^0({\rightmod {\mathcal{A}}})/\acyc {\mathcal{A}}$. 

The derived category can be constructed on the DG level. 
An ${\mathcal{A}}$-module $P$ is \emph{$h$-projective} (resp.~\emph{$h$-flat}) 
if $\homm_{\mathcal{A}}(P,C)$ (resp. $P \otimes_{\mathcal{A}} C$) is an acyclic complex
of ${\mathbb{k}}$-modules for any acyclic $C \in {\rightmod {\mathcal{A}}}$ (resp. $C \in {\leftmod {\mathcal{A}}}$). 
We denote by ${\hproj({\mathcal{A}})}$ the full subcategory of ${\rightmod {\mathcal{A}}}$ consisting of
$h$-projective modules. It follows from the definition that
in ${\hproj({\mathcal{A}})}$ every quasi-isomorphism is a homotopy equivalence, 
and therefore we have $D({\mathcal{A}}) \simeq H^0({\hproj({\mathcal{A}})})$. Alternatively, 
one uses Drinfeld quotients: $D({\mathcal{A}}) = H^0({\rightmod {\mathcal{A}}} / \acyc A)$. 

An object $a$ of triangulated category ${\mathcal T}$ is compact if
$\homm_{\mathcal T}(a,-)$ commutes with infinite direct sums. We write $D_c({\mathcal{A}})$
for the full subcategory of $D({\mathcal{A}})$ consisting of all compact objects. 
We say that an ${\mathcal{A}}$-module $E$ is \emph{perfect} if the class $E \in D({\mathcal{A}})$ is in $D_c({\mathcal{A}})$.

Let ${\mathcal{A}}$ and ${\mathcal{B}}$ be DG categories and let $M$ be an
${\mathcal{A}}$-${\mathcal{B}}$-bimodule. 
We say that $M$ is ${\mathcal{A}}$-perfect (resp. ${\mathcal{B}}$-perfect) 
if it is perfect levelwise in ${\mathcal{A}}$ (resp. ${\mathcal{B}}$). That is, 
$\leftidx{_{a}}{M}{}$ (resp. $M_b$) is a perfect module for all $a \in {\mathcal{A}}$
(resp. $b \in {\mathcal{B}}$). Similarly, for other properties of modules
such as $h$-projective, $h$-flat, or representable. 

Let ${\mathcal{A}}$ be a DG-category. We say that ${\mathcal{A}}$ is \em smooth \rm if the diagonal bimodule ${\mathcal{A}}$ is perfect as an ${\mathcal{A}}$-${\mathcal{A}}$-bimodule. We say that ${\mathcal{A}}$ is
\em proper \rm 
if the 
total cohomology of each of its $\homm$-complexes is
finitely-generated and $D({\mathcal{A}})$ is compactly generated. See  
\cite[Section~2.2]{ToenVaquie-ModuliOfObjectsInDGCategories} for further
details on these two notions. 

\subsection{Restriction and extension of scalars}
\label{section-restriction-and-extension-of-scalars}

Let ${\mathcal{A}}$ and ${\mathcal{B}}$ be two DG-categories and let $M$ be an
${\mathcal{A}}$-${\mathcal{B}}$-bimodule. Moreover, let ${\mathcal{A}}'$ and ${\mathcal{B}}'$ be another two
DG-categories and let $f\colon {\mathcal{A}}' \rightarrow {\mathcal{A}}$ and $g\colon
{\mathcal{B}}' \rightarrow {\mathcal{B}}$ be DG-functors. Define the \em restriction of
scalars of $M$ along $f$ and $g$ \rm to be the ${\mathcal{A}}'$-${\mathcal{B}}'$-bimodule
$\leftidx{_f}{M}{_g}$ defined as $M \circ (f \otimes_{\mathbb{k}} g)$. 
In particular, for any $a \in {\mathcal{A}}$ and $b \in {\mathcal{B}}$ we have
$\leftidx{_a}{(\leftidx{_f}{M}{_g})}{_b} = 
\leftidx{_{f(a)}}{M}{_{g(b)}}$. We write $\leftidx{_f}{M}$ and 
$M_g$ for $\leftidx{_f}{M}{_{\id}}$ and $\leftidx{_{\id}}{M}{_g}$,
respectively. 

Let ${\mathcal{A}}$ and ${\mathcal{B}}$ be two DG-categories and let $f: {\mathcal{A}} \rightarrow {\mathcal{B}}$ 
be a DG-functor. We then have the following three induced functors:
\begin{enumerate}
\item The \em extension of scalars \rm functor
\begin{equation*}
f^*\colon {\rightmod {\mathcal{A}}} \rightarrow {\rightmod {\mathcal{B}}}, 
\end{equation*}
is defined to be $(-) \otimes_{\mathcal{A}} \leftidx{_f}{{\mathcal{B}}}$. 
\item The \em restriction of scalars \rm functor
\begin{equation*}
	f_*\colon {\rightmod {\mathcal{B}}} \rightarrow {\rightmod {\mathcal{A}}}, 
\end{equation*}
is defined to be $(-) \otimes_{\mathcal{B}} {\mathcal{B}}_f$. It sends each $E \in {\rightmod {\mathcal{B}}}$ to its
restriction $E_f \in {\rightmod {\mathcal{A}}}$, and therefore sends acyclic modules to
acyclic modules. 
\item The \em twisted extension of scalars \rm functor
\begin{equation*}
f^!\colon {\rightmod {\mathcal{A}}} \rightarrow {\rightmod {\mathcal{B}}}, 
\end{equation*}
is defined to be $\homm_{{\mathcal{A}}}({\mathcal{B}}_f,-)$. 
\end{enumerate}
By Tensor-$\homm$ adjunction $(f^*, f_*)$ and $(f_*, f^!)$ 
are adjoint pairs of DG-functors. Since $f_*$ preserves acyclic modules 
it follows that $f^*$ preserves $h$-projective modules and $f^!$
preserves $h$-injective modules. Eventually, we obtain the derived functors (whose derivedness we omit to denote in the rest of the paper):
\[ f^\ast:=Lf^*\colon D({\mathcal{A}}) \rightarrow D({\mathcal{B}}),\quad f_*\colon D({\mathcal{B}}) \rightarrow D({\mathcal{A}}), \quad    f^!:=Rf^!\colon D({\mathcal{A}}) \rightarrow D({\mathcal{B}}). \]
Again these form an adjoint triple.


\section{Recollement}
\label{sec:recoll}

The notion of a recollement of triangulated categories was introduced in \cite[Section 1.4]{deligne1983faisceaux}. Let $\mathcal{S}$, ${\mathcal T}$ and $\mathcal{Q}$ be triangulated categories. 
A recollement is a diagram of triangulated functors
  \[
\begin{tikzpicture}
	\node (v1) at (0,0)  {$\mathcal{Q}$};
	\node (v2) at (3,0)  {${\mathcal T}$};
		\node (v3) at (6,0)  {$\mathcal{S}$};
	
	\draw [->] (v1) to node[above] {$q_{\ast}$} (v2);
	\draw [->,bend right=30] (v2) to node[above] {$\scriptstyle{q^{\ast}}$} (v1);
	\draw [->,bend left=30] (v2) to node[below] {$\scriptstyle{q^{!}}$} (v1);
	\draw [->] (v2) to node[above] {$i_{\ast}$} (v3);
	\draw [->,bend right=30] (v3) to node[above] {$\scriptstyle{i^{\ast}}$} (v2);
	\draw [->,bend left=30] (v3) to node[below] {$\scriptstyle{i^{!}}$} (v2);
\end{tikzpicture}
\]
such that 
\begin{enumerate}
\item $(q^{\ast},q_{\ast}), (q_{\ast},q^{!}),(i^{\ast},i_{\ast}),(i_{\ast},i^{!})$ are adjoint pairs;
\item $q_{\ast}$, $i^{!}$ and $i^{!}$ are fully faithful;
\item the composite of two functors in each row is zero.
\item there are two canonical triangles for each object $X$ of ${\mathcal T}$:
\[ i^{\ast} i_{\ast} X \to X \to q_{\ast}q^{\ast} X \to i^{\ast} i_{\ast} X[1] \]
and
\[ q^{!} q_{\ast} X \to X \to i^{!}i_{\ast} X \to q^{!} q_{\ast}  X[1]. \]
\end{enumerate}

We show that a recollement arises from taking a Drinfeld quotient of dg enhanced triangulated categories.
Let  ${\mathcal I} $ be a strictly full dg subcategory of a DG category ${\mathcal V}$. Denote by $i:  {\mathcal I} \to {\mathcal V}$ the inclusion functor and by $q: {\mathcal V} \to {\mathcal V}/{\mathcal I}$ the quotient functor to the Drinfeld quotient.
\begin{Theorem}
\label{thm:1}
Let ${\mathcal I}$ be a strictly full dg subcategory of ${\mathcal V}$. Then there is a recollement
  \[
\begin{tikzpicture}
	\node (v1) at (0,0)  {$D({\mathcal V}/{\mathcal I})$};
	\node (v2) at (3,0)  {$D({\mathcal V})$};
		\node (v3) at (6,0)  {$D({\mathcal I})$};
	
	\draw [->] (v1) to node[above] {$q_{\ast}$} (v2);
	\draw [->,bend right=30] (v2) to node[above] {$\scriptstyle{q^{\ast}}$} (v1);
	\draw [->,bend left=30] (v2) to node[below] {$\scriptstyle{q^{!}}$} (v1);
	\draw [->] (v2) to node[above] {$i_{\ast}$} (v3);
	\draw [->,bend right=30] (v3) to node[above] {$\scriptstyle{i^{\ast}}$} (v2);
	\draw [->,bend left=30] (v3) to node[below] {$\scriptstyle{i^{!}}$} (v2);
\end{tikzpicture}.
\]
In particular $i^{\ast}$, $i^{!}$ and $q_{\ast}$ are fully faithful functors.
\end{Theorem}
\begin{proof}
We adapt the proof of \cite[Proposition on page 2]{jorgensen2009new} to our setting.
As $i^{\ast}$ is the inclusion of $D({\mathcal I})$ into $D({\mathcal V})$ (see  \cite[Proposition~1.15]{lunts2010uniqueness}), it is fully faithful and triangulated. Therefore, $i_{\ast}$ is also triangulated by \cite[Lemma 5.3.6]{neeman2001triangulated}. As $i^{\ast}$ sends compact objects to compact objects (see for example the proof of \cite[Proposition~1.15]{lunts2010uniqueness}), $i_{\ast}$ respects set indexed coproducts by \cite[Theorem 5.1]{neeman1996grothendieck}. This is exactly the setting of \cite{miyachi1991localization} which gives a recollement on the triple $(\mathrm{Ker}(i_{\ast}), D({\mathcal V}),D({\mathcal I}))$. 

For $E$ to be in $\mathrm{Ker}(i_{\ast})$ means that $i_{\ast} E=0$. This holds precisely when $\Hom_{D({\mathcal I})}(I,i_{\ast}E)$ for all $I \in {\mathcal I}$. But
\[ \Hom_{D({\mathcal I})}(I,i_{\ast}E)=\Hom_{D({\mathcal V})}(i^{\ast}I,E). \]
As $i^{\ast}$ is fully faithful, this means that $\mathrm{Ker}(i_{\ast})=D({\mathcal I})^{\perp}$, the right orthogonal to $D({\mathcal I})$ in $D({\mathcal V})$. Then \cite[Proposition~4.9.1~(5)]{krause2008localization} shows that the composition \[D({\mathcal I})^{\perp} \to D({\mathcal V}) \to D({\mathcal V})/D({\mathcal I})=D({\mathcal V}/{\mathcal I})\] is an equivalence.

\end{proof}

\begin{Theorem} 
\label{thm:2}
Let ${\mathcal I}$ be a strictly full dg subcategory of ${\mathcal V}$. The following conditions are equivalent:
\begin{enumerate}
\item ${\mathcal V}_i$ is a perfect ${\mathcal I}$-module.
\item $i_{\ast}$ preserves compact objects.
\item 	$q_{\ast}$ preserves compact objects.
\item $({\mathcal V}/{\mathcal I})_q$ is a perfect ${\mathcal V}$-module.
\end{enumerate}
If this is the case, there is a ``half recollement''
 \[
\begin{tikzpicture}
	\node (v1) at (0,0)  {$D_c({\mathcal V}/{\mathcal I})$};
	\node (v2) at (3,0)  {$D_c({\mathcal V})$};
	\node (v3) at (6,0)  {$D_c({\mathcal I})$};
	
	\draw [->] (v1) to node[above] {$\scriptstyle{q_{\ast}}$} (v2);
	\draw [->,bend right=30] (v2) to node[above] {$\scriptstyle{q^{\ast}}$} (v1);
	\draw [->] (v2) to node[above] {$\scriptstyle{i_{\ast}}$} (v3);
	\draw [->,bend right=30] (v3) to node[above] {$\scriptstyle{i^{\ast}}$} (v2);
\end{tikzpicture}
\]
\end{Theorem}
\begin{proof}
First we show the equivalence of (1) and (2). 
Our setting satisfies the assumptions for Brown representability \cite[Theorem 4.1]{neeman1996grothendieck}. By \cite[Theorem 5.1]{neeman1996grothendieck}, $i_{\ast}: D({\mathcal V}) \to D({\mathcal I})$ preserves compactness if and only if its right adjoint $i^{!}(-)=R\Hom_{{\mathcal I}}({\mathcal V}_i,-)$ preserves arbitrary direct sums. As the $h$-injective resolution of a direct sum is the direct sum of the $h$-injective resolution of the summands, $R\Hom_{{\mathcal I}}({\mathcal V}_i,-)$ preserves arbitrary direct sums if and only if $\Hom_{{\mathcal I}}({\mathcal V}_i,-)$ preserves arbitrary direct sums. By definition, this happens if ${\mathcal V}_i$ is a perfect ${\mathcal I}$-module.

The proof of the equivalence of (3) and (4) is similar. 

Finally, the equivalence of (2) and (3) follows from \cite[Lemma 2.2]{chen2012recollements}, because the derived category of a DG category is compactly generated.
\end{proof}

Recall that a functor $h\colon {\mathcal T} \to Vect$ from an Ext-finite triangulated category ${\mathcal T}$ to the category of vector spaces is called cohomological if it takes exact triangles to long exact sequences. It is of finite type if for every object $A \in {\mathcal T}$, $\mathrm{dim}\oplus_n h(A[n]) < \infty$. For every object $A \in {\mathcal T}$ the functors $h_A(\cdot):=\Hom(A,\cdot)$ and $h^A(\cdot):=\Hom(\cdot,A)$ are finite dimensional cohomological functors by definition. Covariant (resp. contravariant) cohomological functors isomorphic to $h_A$ (resp. $h^A$) are called representable. The category ${\mathcal T}$ is called left (resp. right) saturated if every covariant (resp. contravariant) cohomological functor of finite type is representable. Right saturatedness gives a sufficient condition for the existence of the half recollement of Theorem~\ref{thm:2}.

\begin{Lemma} Let ${\mathcal I}$ be a strictly full dg subcategory of ${\mathcal V}$.
\begin{enumerate}
	\item If $D_c({\mathcal V})$ is right saturated, then $q_{\ast}$ preserves compact objects.
	\item If $D_c({\mathcal I})$ is right saturated, then $i_{\ast}$ preserves compact objects.
\end{enumerate}	
\end{Lemma}
\begin{proof}
We just prove (1). As adjoint functors are unique up to an isomorphism, it is enough to show that $q^{\ast}: D_c({\mathcal V}) \to D_c({\mathcal V}/{\mathcal I})$ has a right adjoint. For any $b \in D_c({\mathcal V}/{\mathcal I})$ consider the functor
\[F_b(a)=\Hom(q^{\ast}(a),b).\] 
This is a contravariant cohomological functor of finite type. Hence, it is representable by an object $q'_{\ast}(b) \in D_c({\mathcal V})$:
\[ F_b(a)=\Hom(a,q'_{\ast}(b)). \]
As $\Hom$ is functorial in both variables, the correspondence $b \mapsto q'_{\ast}(b)$ is also functorial. By the construction, $q'_{\ast}$ is right adjoint to $q^{\ast}$, so it must agree with the restriction of $q_{\ast}$ to $D_c({\mathcal V}/{\mathcal I})$.
\end{proof}

\section{An exact sequence for numerical Grothendieck groups}\label{subsec:prelim_grothendieck_group}\label{subsec:heisenberg_algebra_cat}

Consider a dg category ${\mathcal V}$.
The Grothendieck group $K_0({{\mathcal V}}) = K_0(D_c({\mathcal V}))$ of ${\mathcal V}$ comes equipped with the \emph{Euler} (or \emph{Mukai}) \emph{pairing}
\[
    [a],\, [b] \mapsto \langle [a],[b]\rangle_{\chi} \coloneqq \chi(\Hom_{\mathcal V}(a,b)) =
    \sum_{n\in \mathbb{Z}} (-1)^n  H^n\Hom_{\mathcal V}(a,b).
\]

Recall that a covariant autoequivalence $S: D_c({\mathcal V}) \to D_c({\mathcal V})$ is a Serre functor, if there is a bifunctorial isomorphism
\[ \Hom(A,B)^\ast \cong \Hom(B,SA),\; A,B \in D_c({\mathcal V}). \]

\begin{Example}
\label{ex:smoothprop}
\begin{enumerate}
\item	If ${\mathcal V}$ is smooth and proper, then $D_c({\mathcal V})$ has a Serre functor  \cite[Section~3]{BondalKapranov:1989:RepresentableFunctorsSerreFunctors}.
\item Let $A$ be a finite-dimensional ${\mathbb{k}}$-algebra of finite global dimension, and ${\mathcal V}$ the DG category of right (or left) $A$-modules. Then $D_c({\mathcal V})$, the bounded derived category of finite-dimensional $A$-modules has a Serre functor.
\end{enumerate}
\end{Example}

\begin{Proposition}\label{prop:tabuada_Ggp}
 Suppose that $D_c({\mathcal V})$ has a Serre functor.
    \begin{enumerate}
        \item For every pair of objects $a,\, b$ of $D_c({\mathcal V})$
            \[
                \langle [a],[b]\rangle_{\chi} = \langle [b],[Sa]\rangle_{\chi} =  \langle [S^{-1}b], [a]\rangle_{\chi}
            \]
            where $S$ is the Serre functor  on $D_c({\mathcal V})$.
        \item The left and right kernels of $\chi$ agree.
    \end{enumerate}
\end{Proposition}

\begin{proof}
For ${\mathcal V}$ smooth and proper this statement was proved in Lemma~4.25 and Proposition~4.24 of \cite{tabuada2015noncommutative}. But the proof in fact only uses the existence of the Serre functor.
\end{proof}

\begin{Corollary} Suppose that ${\mathcal V}$ is a DG category such that the left and right kernels of $\chi$ agree (e.g. if $D_c({\mathcal V})$ has a Serre functor). Then there is a well-defined \emph{numerical Grothendieck group} $\numGgp{{\mathcal V}} \coloneqq K_0({\mathcal V})/\ker(\chi)$ of ${{\mathcal V}}$.
\end{Corollary}

\begin{Proposition}[{\cite[Theorem 1.2]{tabuada2016noncommutative}, \cite[Theorem 1.2]{tabuada2017finite}}]
    The numerical Grothendieck group $\numGgp{{\mathcal V}}$ is a finitely generated free abelian group.
\end{Proposition}

If $f\colon {\mathcal{A}} \to {\mathcal{B}}$ is a functor of dg categories, then $f^*\colon D_c(\mathcal A) \to D_c(\mathcal B)$ induces a morphism $K^0({\mathcal{A}}) \to K^0({\mathcal{B}})$.

\begin{Lemma}
\label{lem:indmapnum}
    Let ${\mathcal{A}}, {\mathcal{B}}$ be DG-categories whose numerical Grothendieck group exist, and let $f\colon {\mathcal{A}} \to {\mathcal{B}}$ be a DG-functor.
    Then $f^*$ induces a map
    \[
        f^\ast\colon \numGgp{ {\mathcal{A}}} \to \numGgp{ {\mathcal{B}}}.
    \]
    Further, if $f$ preserves compactness, then it induces a map
    \[
        f_\ast\colon \numGgp{ {\mathcal{B}}} \to \numGgp{ {\mathcal{A}}}.
    \]
\end{Lemma}


\begin{proof}
    As the Euler pairing is defined in terms of Hom-spaces in the derived category, it is compatible with standard derived adjunctions. 
    For any $a \in D_c({\mathcal{A}})$ and for any $b \in D_c({\mathcal{B}})$ we have
    \begin{gather*}
        \chi(f^\ast(a), b) =
        \sum (-1)^i \dim \Hom^i_{D_c({\mathcal{B}})}(f^\ast(a), b) = 
        \sum (-1)^i \dim \Hom^i_{D_c({\mathcal{A}})}(a,f_\ast(b) ).
    \end{gather*}
    While $f_*(b)$ is not necessarily compact, it is a (homotopy) colimit of compact objects $c_i \in D_c({\mathcal{A}})$ in the sense of \stackcite{090Z}.
    As $a$ is compact, $\Hom_{D_c({\mathcal{A}})}(a, -)$ commutes with homotopy colimits.
    Thus the above is equal to
    \[
        \sum (-1)^i \dim \colim \Hom^i_{D_c({\mathcal{A}})}(a, c_i).
    \]
    Computing the Euler characteristic is compatible with taking direct sums.
    Thus for $a \in \ker \chi$ this expression vanishes, and $f^*(a)$ is again in $\ker \chi$ as required.

    The second statement follows directly from the adjunction
    \[
        \Hom_{D_c({\mathcal{A}})}(a, f_\ast(b)) = 
        \Hom_{D_c({\mathcal{B}})}(f^\ast(a), b).
        \qedhere
    \]
 \end{proof}


\begin{Lemma} 
\label{lem:kermaps}
	Let ${\mathcal V}$ a dg category and let ${\mathcal I}$ be a strictly full dg subcategory of ${\mathcal V}$. 
 Assume that the numerical Grothendieck group of ${\mathcal V}$, ${\mathcal I}$ and ${\mathcal V}/{\mathcal I}$ exist. Then:
\begin{enumerate}
	\item $i^\ast(K_0({\mathcal I})) \cap \mathrm{Ker}(\chi_{{\mathcal V}}) = i^\ast(\mathrm{Ker}(\chi_{{\mathcal I}}))$;
	\item $\mathrm{Ker}(\chi_{{\mathcal V}/{\mathcal I}})\supseteq q^\ast(\mathrm{Ker}(\chi_{{\mathcal V}}))$.
\end{enumerate}
If moreover $q$ or, equivalently, $i$ preserves compactness, then 
\begin{enumerate}
	\item[(3)] $\mathrm{Ker}(\chi_{{\mathcal V}/{\mathcal I}})=  q^\ast(\mathrm{Ker}(\chi_{{\mathcal V}}))$.
\end{enumerate}
\end{Lemma}
\begin{proof}
	(1): The containment ``$\supseteq$'' follows from the proof of Lemma~\ref{lem:indmapnum}.
	For ``$\subseteq$''  we observe that if $i^\ast(a) \in \mathrm{Ker}(\chi_{{\mathcal V}})$, then for every $b \in  K_0({\mathcal I})$ we have that $\chi(i^{\ast}a,i^\ast b)=0$. But as $i^\ast$ is the inclusion of a full subcategory,
	the same equation holds in  $K_0({\mathcal I})$.
	
	(2): Again, follows from the proof Lemma~\ref{lem:indmapnum}. 

	(3): Let $a \in \mathrm{Ker}(\chi_{{\mathcal V}/{\mathcal I} })$ and let $b:=q_{\ast}(a)$. Then for any $c \in D_c({\mathcal V})$:
	\[ \chi(b,c)=\chi(q_{\ast}(a),c)=\chi(a,q^{\ast}(c))=0. \]
	Hence, $b \in \mathrm{Ker}(\chi_{{\mathcal V}})$. 
 As $q^\ast \circ q_\ast$ is the identity of on objects of $D_c({\mathcal V}/{\mathcal I})$, we have that $q^\ast(b)=q^\ast(q_{\ast}(a))=a$.

\end{proof}

As mentioned in the Section~\ref{sec:intro}, there is an exact sequence of the ordinary Grothendieck groups \cite[Proposition~VIII.3.1.]{grothendieck1977cohomologie}:
\begin{equation}  
\label{eq:kgrpseq2}
K_0({\mathcal I}) \xrightarrow{i^\ast} K_0({\mathcal V}) \xrightarrow{q^\ast} K_0({\mathcal V}/{\mathcal I})  \to 0.
\end{equation}
We now descend this sequence to numerical Grothendieck groups. 

\begin{Theorem}
\label{lem:numgrpseq}
Let ${\mathcal V}$ and let ${\mathcal I}$ be a strictly full dg subcategory of ${\mathcal V}$. Assume that the numerical Grothendieck group of ${\mathcal V}$, ${\mathcal I}$ and ${\mathcal V}/{\mathcal I}$ exist.
Suppose moreover that $q$ or, equivalently, $i$ preserves compactness.  
Then there is an exact sequence
\[\numGgp{ {\mathcal I}} \to \numGgp{ {\mathcal V}} \to \numGgp{{\mathcal V}/{\mathcal I}} \to 0.\]
\end{Theorem}
\begin{proof}
 
It follows from \eqref{eq:kgrpseq2} that $K_0({\mathcal V}/{\mathcal I})=K_0({\mathcal V}) / i^\ast(K_0({\mathcal I}))$. For short, denote 
\[\mathcal{K}_{{\mathcal V}} \coloneqq \mathrm{Ker}(\chi_{{\mathcal V}})\]
and similarly for ${\mathcal I}$ and ${\mathcal V}/{\mathcal I}$.
Then by Lemma~\ref{lem:kermaps}~(1) we have
\[ 
\begin{aligned}
\numGgp{ {\mathcal V}}/i^\ast(\numGgp{ {\mathcal I}}) & = (K_0({\mathcal V})/\mathcal{K}_{{\mathcal V}}) / (i^\ast( K_0({\mathcal I})/\mathcal{K}_{{\mathcal I}})\\ & =(K_0({\mathcal V})/\mathcal{K}_{{\mathcal V}}) / (i^\ast( K_0({\mathcal I})) /  i^\ast(\mathcal{K}_{{\mathcal I}})) \\
& \simeq (K_0({\mathcal V})/\mathcal{K}_{{\mathcal V}}) / (i^\ast( K_0({\mathcal I})) /  i^\ast(K_0({\mathcal I})) \cap \mathcal{K}_{{\mathcal V}} ).
\end{aligned}\]
As
\[ i^\ast( K_0({\mathcal I})) /  i^\ast(K_0({\mathcal I})) \cap \mathcal{K}_{{\mathcal V}}\simeq (i^\ast( K_0({\mathcal I})) \cdot \mathcal{K}_{{\mathcal V}}) / \mathcal{K}_{{\mathcal V}},\]
the above quotient is the same as
\[K_0({\mathcal V}) / (i^\ast( K_0({\mathcal I})) \cdot \mathcal{K}_{{\mathcal V}}).\]
Using similarly the identity
\[(i^\ast( K_0({\mathcal I})) \cdot \mathcal{K}_{{\mathcal V}})/i^\ast( K_0({\mathcal I}))\simeq  \mathcal{K}_{{\mathcal V}}/( i^\ast( K_0({\mathcal I})) \cap \mathcal{K}_{{\mathcal V}}) \]
we obtain that
\[\begin{aligned}
K_0({\mathcal V}) / (i^\ast( K_0({\mathcal I})) \cdot \mathcal{K}_{{\mathcal V}}) & \simeq (K_0({\mathcal V})/i^\ast(K_0({\mathcal I}))) / (\mathcal{K}_{{\mathcal V}}  /  i^\ast(K_0({\mathcal I})) \cap \mathcal{K}_{{\mathcal V}} ) \\
&= (K_0({\mathcal V}) / i^\ast( K_0({\mathcal I}))) /  q^\ast(\mathcal{K}_{{\mathcal V}}) \\
& = (K_0({\mathcal V}) / i^\ast( K_0({\mathcal I})))/\mathcal{K}_{{\mathcal V}/{\mathcal I}},
\end{aligned}
 \]
where at the second equality we used the definition of $q^\ast$, and at the last equality we used Lemma~\ref{lem:kermaps}~(3). Noting that 
the last line is $\numGgp{{\mathcal V}/{\mathcal I}}$, we deduce the statement.








\end{proof}


\begin{Corollary}
\label{lem:numgrpseq2}
Let ${\mathcal V}$ and let ${\mathcal I}$ be a strictly full dg subcategory of ${\mathcal V}$. Assume that the numerical Grothendieck group of ${\mathcal V}$, ${\mathcal I}$ and ${\mathcal V}/{\mathcal I}$ exist.
Suppose moreover that $K_0({\mathcal V}/{\mathcal I})$ has no torsion. 
Then there is an exact sequence
\[\numGgp{ {\mathcal I}} \to \numGgp{ {\mathcal V}} \to \numGgp{{\mathcal V}/{\mathcal I}} \to 0.\]
\end{Corollary}
\begin{proof}
	Let $a \in K_0({\mathcal V}/{\mathcal I})$, and let $A$ be the subgroup generated by $a$. If $A$ is of infinite order, then because the cokernel of $i^{\ast}$ on level of the K-groups $\mathrm{Coker}(i^{\ast}) \cong K_0({\mathcal V}/{\mathcal I})$, we must have that \[\mathrm{Im}(i^{\ast}) \cap (q^{\ast})^{-1}(A)=0.\] This means that $q^{\ast}$ maps $(q^{\ast})^{-1}(A)$ to $A$ bijectively. As the quotient functor $q: {\mathcal V} \to {\mathcal V}/{\mathcal I}$ is the identity on objects, the same is true for $q^\ast$. 
When $A$ is of infinite order, the above considerations mean that $q^\ast$ does not introduce any new relation for the elements representing the classes in $(q^{\ast})^{-1}(A)$. Such new relations could occur only from the new morphisms introduced when quotienting. Recall that morphisms in $D_c({\mathcal V}/{\mathcal I})$ are roofs $b \leftarrow c \rightarrow d$ such that $Cone(c \to b) \in D_c({\mathcal I})$, but these morphisms are only considered up to an equivalence relation. Because there are no new relations, we must have that there are no new morphisms in the quotient. That is, all the new roofs are killed by the equivalence relation when $b \in (q^{\ast})^{-1}(A)$ or $e \in (q^{\ast})^{-1}(A)$.
As a consequence,
\[ \Hom_{D_c({\mathcal V})}(b,c)=\Hom_{D_c({\mathcal V}/{\mathcal I})}(q^\ast b,q^\ast c) \]
for any $b \in (q^{\ast})^{-1}(A)$ and $c \in K_0({\mathcal V})$.
Therefore, if $a \in \mathrm{Ker}(\chi_{{\mathcal V}/{\mathcal I} })$, then $(q^{\ast})^{-1}(a) \in \mathrm{Ker}(\chi_{{\mathcal V}})$ because
\[ \chi((q^{\ast})^{-1}(a),c)=\chi(a,q^{\ast}c)=0 \]
for every $c \in K_0({\mathcal V})$.
In particular, if $K_0({\mathcal V}/{\mathcal I})$ has no torsion, then the conclusion of Lemma~\ref{lem:kermaps}~(3) holds, even without the assumption that $q$ or $i$ preserve compactness. The statement then follows as in Theorem~\ref{lem:numgrpseq}.
\end{proof}

\bibliographystyle{amsplain}
\bibliography{knumseq}
\end{document}